\documentclass[12pt]{article} 
\usepackage{amsmath, amsthm, amssymb, graphicx, hyperref, xcolor}
\hypersetup{colorlinks, linkcolor={black!90!black}, citecolor={blue!50!black}, urlcolor={blue!80!black}}
\usepackage[british]{babel}
\usepackage[a4paper, margin=3cm]{geometry} 
\numberwithin{equation}{section}

\newtheorem{lemma}{Lemma}[section]
\newtheorem{theorem}{Theorem}[section]

\title{Asymptotic height distribution in high-dimensional sandpiles}
\author{Antal A. J\'{a}rai \and Minwei Sun\thanks{Department of
Mathematical Sciences, University of Bath, Claverton Down, Bath BA2 7AY, United Kingdom \newline
Email: \tt{A.Jarai@bath.ac.uk}, \tt{M.Sun@bath.ac.uk}}} 
\date{25th October 2019}

\def\Pr{\mathbf{P}}
\def\E{\mathbf{E}}

\def\Z{\mathbb{Z}}
\def\Poisson{\mathsf{Poisson}}
\def\USF{\mathsf{USF}}
\def\UST{\mathsf{UST}}
\def\cT{\mathcal{T}}

\begin{document}
\maketitle

\begin{abstract}
We give an asymptotic formula for the single site height distribution
of Abelian sandpiles on $\Z^d$ as $d \to \infty$, in terms of 
$\Poisson(1)$ probabilities. We provide error estimates.
\end{abstract}

\section{Introduction}
\label{sec:introduction}

We consider the Abelian sandpile model on the nearest neighbour lattice $\Z^d$;
see Section \ref{sec:background} for definitions and background.
Let $\Pr$ denote the weak limit of the stationary distributions $\Pr_L$
in finite boxes $[-L,L]^d \cap \Z^d$.
Let $\eta$ denote a sample configuration from the measure $\Pr$.
Let $p_d(i) = \Pr[\eta(o) = i]$, $i = 0, \dots, 2d-1$, denote the 
height probabilities at the origin in $d$ dimensions.
The following theorem is our main result that states the asymptotic form of 
these probabilities as $d \to \infty$. 

\begin{theorem}
\label{thm:asymformula}
(i) For $0 \le i \le d^{1/2}$, we have 
\begin{equation}
\label{e:formula}
  p_d(i) 
  = \sum_{j=0}^i \frac{e^{-1} \frac{1}{j!}}{2d-j} + O\Big(\frac{i}{d^2}\Big)
  = \frac{1}{2d} \sum_{j=0}^i e^{-1} \frac{1}{j!} + O\Big(\frac{i}{d^2}\Big).
\end{equation}
(ii) If $d^{1/2} < i \le 2d-1$, we have
\begin{equation*}
  p_d(i)
  = p_d(d^{1/2}) + O(d^{-3/2}).
\end{equation*}
In particular, $p_d(i) \sim (2d)^{-1}$, if $i,d \to \infty$.
\end{theorem}

The appearance of the $\Poisson(1)$ distribution in the above formula
is closely related to the result of Aldous \cite{Aldous1990}
that the degree distribution of the origin in the uniform spanning forest
in $\Z^d$ tends to $1$ plus a $\Poisson(1)$ random variable as $d \to \infty$.
Indeed our proof of \eqref{e:formula} is achieved by showing that in the 
uniform spanning forest of $\Z^d$, the number of neighbours $w$ of the origin $o$,
such that the unique path from $w$ to infinity passes through $o$ is
asymptotically the same as the degree of $o$ minus $1$, that is, $\Poisson(1)$.

In \cite{JaraiSun2019a} we compared the formula \eqref{e:formula} to numerical 
simulations in $d = 32$ on a finite box with $L = 128$, and there is excellent
agreement with the asymptotics already for these values.

Other graphs where information on the height distribution is available are as follows.
Dhar and Majumdar \cite{DharMajumdar1990} studied the Abelian sandpile model 
on the Bethe lattice and the exact expressions for various distribution functions 
including the height distribution at a vertex were obtained using combinatorial methods.
For the single site height distribution they obtained (see \cite[Eqn.~(8.2)]{DharMajumdar1990})
\begin{equation*}
 p_{\mathrm{Bethe},d}(i)
 = \frac{1}{(d^2 - 1) \, d^d} \sum_{j = 0}^i \binom{d+1}{j} (d-1)^{d-j+1}. 
\end{equation*}
If one lets the degree $d \to \infty$ in this formula, one obtains the form
in the right hand side of \eqref{e:formula} for any fixed $i$ (with $2d$ replaced by $d$). 

Exact expressions for the distribution of height probabilities were derived 
by Papoyan and Shcherbakov \cite{PapoyanShcherbakov1996} on the Husimi lattice of 
triangles with an arbitrary coordination number $q$.
However, on $d$-dimensional cubic lattices of $d\geq 2$, exact results for 
the height probability are only known for $d = 2$; see
\cite{MajumdarDhar1991}, \cite{Priezzhev1994}, \cite{JengPirouxRuelle2006},
\cite{KenyonWilson2015}, \cite{PoghosyanPriezzhevRuelle2011}.

\subsection{Definitions and background}
\label{sec:background}
Sandpiles are a lattice model of self-organized criticality, introduced by 
Bak, Tang and Wiesenfeld \cite{BakTangWiesenfeld1987}, and have been studied 
in both physics and mathematics. See the surveys 
\cite{Jarai2018}, \cite{LevinePeres2017},
\cite{Redig2006}, \cite{HolroydLevineMeszarosPeresProppWilson2008}, 
\cite{Dhar2006}.
Although the model can easily be defined on an arbitrary finite connected graph,
in this paper we will restrict to subsets of $\Z^d$.

Let $V_L = [-L,L]^d \cap \Z^d$ be a box of radius $L$, where $L \ge 1$.
For simplicity, we suppress the $d$-dependence in our notation.
We let $G_L = (V_L \cup \{s \},E_L)$ denote the graph obtained from $\Z^d$
by identifying all vertices in $\Z^d \setminus V_L$ that becomes $s$,
and removing loop-edges at $s$. We call $s$ the \emph{sink}.
A \emph{sandpile} $\eta$ is a collection of indistinguishable particles
on $V_L$, specified by a map $\eta : V_L \to \{ 0, 1, 2, \dots \}$. 

We say that $\eta$ is stable at $x\in V_L$, if $\eta (x) < 2d$.
We say that $\eta$ is stable, if $\eta (x) < 2d$, for all $x\in V_L$.
If $\eta$ is unstable (i.e. $\eta (x) \geq 2d$ for some $x\in V_L$),
$x$ is allowed to topple which means that $x$ passes one particle 
along each edge to its neighbours. When the vertex $x$ topples, 
the particles are re-distributed as follows:
\begin{equation*}
\begin{split}
&\eta (x) \rightarrow \eta(x) - 2d; \\
&\eta (y) \rightarrow \eta(y) + 1, \quad y \in V_L, y \sim x.
\end{split}
\end{equation*}

Particles arriving at $s$ are lost, so we do not keep track of them. 
Toppling a vertex may generate further unstable vertices. 
Given a sandpile $\xi$ on $V_L$, we define its stabilization
\begin{equation*}
  \xi^\circ \in \Omega_L
  := \{\textsl{all stable sandpiles on $V_L$}\} 
  = \{0,1,\dots,2d-1\}^{V_L}
\end{equation*}
by carrying out all possible topplings, in any order, 
until a stable sandpile is reached.
It was shown by Dhar \cite{Dhar1990} that
the map $\xi \mapsto \xi^\circ$ is well-defined, that is,
the order of topplings does not matter.

We now define the sandpile Markov chain.
The state space is the set of stable sandpiles $\Omega_L$. 
Fix a positive probability distribution $p$ on $V_L$, 
i.e. $\sum_{x\in V_L} p(x) = 1 $ and $p(x) > 0$ for all $x\in V_L$.
Given the current state $\eta \in \Omega_L$, 
choose a random vertex $X \in V$ according to $p$, 
add one particle at $X$ and stabilize. 
The one step transition of the Markov chain moves from
$\eta$ to $(\eta + \bf{1}_X)^{\circ}$.
Considering the sandpile Markov chain on $G_L$,
there is only one recurrent class \cite{Dhar1990}. 
We denote the set of recurrent sandpiles by $\mathcal{R}_L$.
It is known \cite{Dhar1990} that the invariant distribution 
$\Pr_{L}$ of the Markov chain is uniformly distributed on $\mathcal{R}_L$.

Majumdar and Dhar \cite{MajumdarDhar1992} gave a bijection
between $\mathcal{R}_L$ and spanning trees of $G_L$. This maps
the uniform measure $\Pr_L$ on $\mathcal{R}_L$ to the uniform
spanning tree measure $\UST_L$. A variant of this bijection 
was introduced by Priezzhev \cite{Priezzhev1994}, 
and is described in more generality in 
\cite{JaraiWerning2014}, \cite{GamlinJarai2014}. The latter bijection 
enjoys the following property, that we will exploit in this paper.
Orient the spanning tree towards $s$, and let $\pi_L(x)$ denote
the oriented path from a vertex $x$ to $s$. Let
\begin{equation*}
 W_L
 = \{ x \in V_L : o \in \pi_L(x) \}. 
\end{equation*}
Then we have that 
\begin{equation}
\label{e:bijection-rel}
  \parbox{10cm}{conditional on $\deg_{W_L}(o) = i$, the height 
  $\eta(o)$ is uniformly distributed over the values
  $i, i+1, \dots, 2d-1$.} 
\end{equation}
This has the following consequence for the height probabilities.
Let $q^L(i) = \UST_L [ \deg_{W_L}(o) = i ]$, $i = 0, \dots, 2d-1$.
Then 
\begin{equation*}
 p^L(i)
 := \Pr_L [ \eta(o) = i ] 
 = \sum_{j=0}^{i} \frac{q^L(j)}{2d - j}. 
\end{equation*}

The measures $\Pr_L$ have a weak limit $\Pr = \lim_{L \to \infty} \Pr_L$
\cite{AthreyaJarai2004}, and hence $p(i) = \lim_{L \to \infty} p^L(i)$ exist,
$i = 0, \dots, 2d-1$. Although the $q^L(i)$ depend on the non-local
variable $W_L$, one also has that $q(i) = \lim_{L \to \infty} q^L(i)$ exist,
$i = 0, \dots, 2d-1$; see \cite{JaraiWerning2014}. In fact, $q(i)$ is given by the
following natural analogue of its finite volume definition. Consider the 
uniform spanning forest measure $\USF$ on $\Z^d$; defined as the weak limit
of $\UST_L$; see \cite[Chapter 10]{LyonsPeres2016}.
Let $\pi(x)$ denote the unique infinite self-avoiding path in the spanning forest
starting at $x$, and let 
\begin{equation*}
 W
 = \{ x \in \Z^d : o \in \pi(x) \}. 
\end{equation*}
Then $q(i) = \USF [ \deg_{W}(o) = i ]$, $i = 0, \dots, 2d-1$.

Therefore, we have 
\begin{equation}
 p(i)
 := \Pr [ \eta(o) = i ] 
 = \sum_{j=0}^{i} \frac{q(j)}{2d - j}. 
\end{equation}

\subsection{Wilson's method} 

Given a finite path $\gamma = [s_0, s_1, ... , s_k] $ in $\Z^d$, we erase loops from $\gamma$ 
chronologically, as they are created. We trace $\gamma$ until the first time $t$, 
if any, when $s_t \in \{s_0, s_1, ..., s_{t-1}\}$, i.e. there is a loop. 
We suppose $s_t = s_i$, for some $i \in \{0,1,...,t-1\}$ and remove the loop 
$[s_i,s_{i+1},...,s_t=s_i]$. Then we continue tracing $\gamma$ and follow the same procedure 
to remove loops until there are no more loops to remove. 
This gives the loop-erasure $\pi = LE(\gamma)$ of $\gamma$, 
which is a self-avoiding path \cite{MadrasSlade2013}.
If $\gamma$ is generated from a random walk process, the loop-erasure of $\gamma$ 
is call the loop-erased random walk (LERW).

When $d \ge 3$, the $\USF$ on $\Z^d$ can be sampled via Wilson's method rooted at
infinity \cite{BenjaminiLyonsPeresSchramm2001}, \cite[Section 10]{LyonsPeres2016}, that is
described as follows. 
Let $s_1, s_2,\dots$ be an arbitrary enumeration of the vertices and let 
$\cT_0$ be the empty forest with no vertices. 
We start a simple random walk $\gamma_n$ at $s_n$ and $\gamma_n$ stops when $\cT_{n-1}$ is hit,
otherwise we let it run indefinitely. 
$LE(\gamma_n)$ is attached to $\cT_{n-1}$ and the resulting forest is denoted by $\cT_n$.
We continue the same procedure until all the vertices are visited. 
The above gives a random sequence of forests $\cT_1 \subset \cT_2 \subset \dots$, 
where $\cT = \cup_{n} \cT_n $ is a spanning forest of $\Z^d$. 
The extension of Wilson's theorem \cite{Wilson1996} to transient infinite
graphs proved in \cite{BenjaminiLyonsPeresSchramm2001} implies that $\cT$ is distributed
as the $\USF$.

\section{Proof of the main theorem}

Let $(S_n^x)_{n\geq 0}$ be a simple random walk started at $x$ (independent 
between $x$'s on $\Z^d$) and let $\pi(x)$ be the path in the USF
from $x$ to infinity. 
We introduce the events:
\begin{equation*}
\begin{split}
E_i &= \Big\{ |\{ w \sim o : \text{$\pi(w)$ passes through $o$}\}| = i \Big\}, \quad
    i = 0, \dots, 2d-1;\\
E_i(x_1,x_2,\dots,x_i) &= \Big \{ \{ w \sim o : \text{$\pi(w)$ passes through $o$}\} 
= \{x_1, x_2,\dots,x_i\} \Big\}.
\end{split}
\end{equation*}
Then recall that
\begin{equation}
\label{e:qdEi}
q_d(i) = \Pr[\deg_W(o) = i]
= \Pr[E_i] = \sum_{\substack{x_1,\dots,x_i 
\sim o \\ \text{distinct}}} \Pr[E_i(x_1,\dots,x_i)].
\end{equation}

\subsection{Preliminary}

\begin{lemma}
\label{lem:returntoo}
We have $\Pr[\text{$S^o_n = o$ for some $n \ge 2$}]
   = O(1/d)$
and $\Pr[\text{$S^o_n = o$ for some $n \ge 4$}]
   = O(1/d^2)$,
as $d \to\infty$.
\end{lemma}

\begin{proof}
Let $\hat{D}(k)$ = $ \frac{1}{d} \sum_{j=1}^{d} \cos(k_j)$, $ k \in [-\pi,\pi]^d$,
be the Fourier transform in $d$ dimensions of the one-step distribution of RW. 
Lemma A.3 in \cite{MadrasSlade2013} states that for all non-negative integers $n$ 
and all $d\geq 1$ we have
\begin{equation*}
\|\hat{D}^n\|_1
= (2\pi)^{-d}\int_{[-\pi,\pi]^d}|\hat{D}(k)^n|d^dk  
\leq (\frac{\pi d}{4n})^{d/2}.
\end{equation*}
Based on above, we have
\begin{equation}
\label{e:rw-terms}
\begin{split}
\Pr[\text{$S^o_n = o$ for some $n \ge 4$}]
& \leq \frac{1}{(2\pi)^d} \sum_{n=4}^{\infty} \int \hat{D}^n(k)dk \\ 
& \leq \frac{1}{(2\pi)^d} \sum_{n=4}^{d-1} \int \hat{D}^n(k)dk
 + \sum_{n=d}^{\infty} \Big(\frac{\pi d}{4n}\Big)^{d/2}.
\end{split}
\end{equation}

Since $\int \hat{D}^4(k) dk$ and $\int \hat{D}^6(k)dk$ state the probability 
that $S^o$ returns to $o$ in $4$ and $6$ steps each, 
by counting the number of ways to return, they are bounded 
by dimension-independent multiples of $1/d^2$ and $1/d^3$ respectively.
We have $\int \hat{D}^n(k) dk = 0$ with odd $n$, 
and for $6 < n \leq d-1$ and $n$ even, we have 
$\int \hat{D}^n(k)dk \leq \int \hat{D}^6(k)dk$.
Hence, 
\begin{equation*}
\frac{1}{(2\pi)^d} \int \hat{D}^n(k)dk = O\Big(\frac{1}{d^3}\Big), 
  \quad 6 \le n \le d-1. 
\end{equation*}

The last sum in \eqref{e:rw-terms} can be bounded as: 
\begin{equation*}
\begin{split}
\Big(\frac{\pi d}{4}\Big)^{d/2} \sum_{n=d}^{\infty} n^{-d/2}
& \leq\Big(\frac{\pi d}{4}\Big)^{d/2} \int_{d-1}^\infty x^{-d/2} dx 
 = \Big(\frac{\pi d}{4}\Big)^{d/2} \frac{(d-1)^{1-\frac{d}{2}}}{d/2-1}\\
&= \Big(\frac{d-1}{d/2-1}\Big)\Big(\frac{d}{d-1}\Big)^{\frac{d}{2}}
 \Big(\frac{\pi}{4}\Big)^{\frac{d}{2}}
\leq Ce^{-cd},
\end{split}
\end{equation*}
since we can take $d > 4$ and $\frac{\pi}{4} < 1$.

Hence, we have the required results
\begin{equation*}
\begin{split}
\Pr[\text{$S^o_n = o$ for some $n \ge 4$}]
&\leq \int \hat{D}^4(k) dk + d \int \hat{D}^6(k)dk + C e^{-cd}\\
&= O\Big(\frac{1}{d^2}\Big) + d\times O\Big(\frac{1}{d^3}\Big)
= O\Big(\frac{1}{d^2}\Big),\\
\Pr[\text{$S^o_n = o$ for some $n \ge 2$}] 
&\leq \Big(\frac{1}{2d}\Big) + \Pr[\text{$S^o_n = o$ for some $n \ge 4$}] 
= O\Big(\frac{1}{d}\Big).
\end{split}
\end{equation*}

\end{proof}

\subsection{Lower bounds}

Let us fix the vertices $x_1, \dots, x_i \sim o$.
Let 
\begin{equation*}
 A_0 = \Big\{ S_1^o \not\in \{  x_1, \dots, x_i \},\, S_n^o \not\in \mathcal{N}
       \text{ for $n \geq 2 $} \Big\},
\end{equation*} 
where $\mathcal{N} = \{ y \in \Z^d : |y| \leq 1\}$.

\begin{lemma}
We have $\Pr[A_0] \geq 1- O(i/d).$
\end{lemma}

\begin{proof}

\begin{equation*}
\Pr[A_0] 
 = \Pr[S_1^o \neq x_1,\dots,x_i]
   \Pr[S_n^o \not\in \mathcal{N}\text{ for $n \geq 2 $} | S_1^o \neq x_1,\dots,x_i].
\end{equation*} 
We have $\Pr[S_1^o \neq x_1, \dots x_i] = 1-O(i/d)$ and the probability for 
the remaining steps is at least $1-O(1/d)$, shown as follows.
The probabilities $\Pr[S_2^o \neq o | S_1^o \neq x_1,\dots,x_i]$ 
and $\Pr[S_3^o \not\in \mathcal{N} | S_2^o \neq o, S_1^o \neq x_1,\dots,x_i]$ 
are both equal to $1-O(1/d)$.
Considering the s.r.w starting at the position $S_3^o$, 
it hits at most three neighbours of $o$ in two further steps, the remaining neighbours
will need at least $4$ steps to hit, so, by Lemma \ref{lem:returntoo}, we have
\begin{equation*}
\begin{split}
\sum_{\text{at most $3$ neighbours $x_j$}} \sum_{k\geq 1} P_{2k}(S_3^o,x_j)
 &\leq O(\frac{1}{d}), \\
\sum_{\text{the remaining neighbours $x_{j'}$}} \sum_{k\geq 2} P_{2k}(S_3^o,x_{j'})
 &\leq O(d)O(\frac{1}{d^2}) 
 = O(\frac{1}{d}),
\end{split}
\end{equation*}
since $P_{2k}(x,y) \leq P_{2k}(o,o)$ for all $x, y$. Therefore, combining above
results together, we get
$\Pr[S_n^o \not\in \mathcal{N}\text{ for $n \geq 2 $} | S_1^o \neq x_1,\dots,x_i]
\geq 1 - O(1/d)$ as required. 
\end{proof}

Let us label the neighbours of $o$ different from $x_1, \dots, x_i$ as
$x_{i+1}, \dots, x_{2d}$, in any order.
On the event $A_0$, the first step of $\pi(o)$ is to a neighbour of $o$ in 
$\{x_{i+1},\dots,x_{2d}\}$ and we could assume $x_{2d}$ to be the first step of $\pi(o)$. 
Then $\pi(o)$ does not visit other vertices in $\mathcal{N} \backslash \{o\}$.
Define $A_j =\{S_1^{x_j} = o\}$ for $j = 1,2,\dots,i$ and then $\Pr[A_j] = 1/2d$.

Using Wilson's algorithm, consider random walks first started at $o, x_1, .., x_i$ and 
then started at $x_{i+1},\dots,x_{2d-1}$. We obtain the following:
\begin{equation}
\label{e:E_i}
\begin{split}
\Pr[E_i(x_1,\dots,x_i)] &\geq \Pr[A_0] \times \prod_{j=1}^i \Pr[A_j] 
\times \Pr[E_i(x_1,..,x_i)|A_0 \cap A_1 \cap \dots \cap A_i] \\
&\geq \Big(1-O\Big(\frac{i}{d}\Big)\Big)\Big(\frac{1}{2d}\Big)^i 
\Pr[E_i(x_1,..,x_i)|A_0 \cap A_1 \cap \dots \cap A_i].
\end{split}
\end{equation}
Define $B_k = \{S_1^{x_k} \neq o, S_n^{x_k} \not\in \{x_1,\dots,x_i\} \text{ for } n\geq 2\}$ 
for $k = i+1,\dots,2d-1$.

\begin{lemma}
$\Pr[B_k] \geq 1 - 1/2d - O(i/d^2)$,
where $i+1 \leq k \leq 2d-1$.
\end{lemma}

\begin{proof}
We have $\Pr[S_1^{x_k} \neq o] = 1 - 1/2d$.
If the first step is not to $o$, the first step could be in one of the 
$e_1,\dots,e_i$ directions, say $e_j$, with probability $i/2d$. 
Then the probability to hit $x_j$ is $1/2d + O(1/d^2)$. 
Hence,the probability that $S^{x_k}$ hits $\{x_1,\dots,x_i\}$ is $O(i/d^2)$. 
\end{proof}

\begin{lemma}
\label{lem:lowerb}
$q_d(i) \geq e^{-1} \frac{1}{i!}(1+O(\frac{i^2}{d})).$
\end{lemma}

\begin{proof}
By \eqref{e:E_i}, we have 
\begin{equation*}
\Pr[E_i(x_1,\dots,x_i)]
\geq \Big(1- O\Big(\frac{i}{d}\Big)\Big)\Big(\frac{1}{2d}\Big)^i
\Big(1-\frac{1}{2d} + O\Big(\frac{i}{d^2}\Big)\Big)^{2d-1-i}.
\end{equation*}
Then by \eqref{e:qdEi},
\begin{equation*}
\begin{split}
q_d(i)
&\geq {2d \choose i}\Big(1- O\Big(\frac{i}{d}\Big)\Big)\Big(\frac{1}{2d}\Big)^i
\Big(1-\frac{1}{2d}+O\Big(\frac{i}{d^2}\Big)\Big)^{2d-1-i}\\
& =\frac{2d(2d-1)\dots(2d-i+1)}{i!(2d)^i}\Big(1-O\Big(\frac{i}{d}\Big)\Big)
\Big(1-\frac{1}{2d}+O\Big(\frac{i}{d^2}\Big)\Big)^{2d}\Big(1+O\Big(\frac{i}{d}\Big)\Big),
\end{split}
\end{equation*}
where 
\begin{equation*}
\begin{split}
\Big(1 - \frac{1}{2d} + O\Big(\frac{i}{d^2}\Big)\Big)^{2d}
&= \exp\Big(2d\times\log\Big(1 - \frac{1}{2d} + O\Big(\frac{i}{d^2}\Big)\Big)
= \exp\Big(2d\Big(- \frac{1}{2d} + O\Big(\frac{i}{d^2}\Big)\Big) \\
&= \exp\Big(-1 + O\Big(\frac{i}{d}\Big)\Big) 
= e^{-1}\Big(1+O\Big(\frac{i}{d}\Big)\Big),
\end{split}
\end{equation*}
and 
\begin{equation*}
\begin{split}
\frac{2d(2d-1)\dots(2d-i+1)}{(2d)^i} 
= 1\Big(1-\frac{1}{2d}\Big)\Big(1-\frac{2}{2d}\Big)\dots\Big(1-\frac{i}{2d}+\frac{1}{2d}\Big)
= \Big(1 + O\Big(\frac{i^2}{d}\Big)\Big).
\end{split}
\end{equation*}
Then the result follows
\begin{equation*}
q_d(i) 
\geq e^{-1} \frac{1}{i!}\Big(1+O\Big(\frac{i}{d}\Big)\Big)
  \Big(1+O\Big(\frac{i^2}{d}\Big)\Big)
  \Big(1+O\Big(\frac{i}{d}\Big)\Big)
  \Big(1-O\Big(\frac{i}{d}\Big)\Big) 
= e^{-1} \frac{1}{i!}\Big(1+O\Big(\frac{i^2}{d}\Big)\Big).
\end{equation*}
\end{proof}

The above lemma gives a lower bound for $q_d$ and we now prove an upper bound.

\subsection{Upper bounds}
Recall that $\pi(o)$ denotes the unique infinite self-avoiding path in the 
spanning forest starting at $o$ and 
let $\bar{A}_o = \{\text{$\pi(o)$ visits only one neighbour of $o$}\}$.
\begin{lemma}
\label{lem:barAo}
$\Pr[\text{$\pi(o)$ visits more than one neighbour of $o$}] 
= P[\bar{A}_o^c] = O(1/d).$
\end{lemma}

\begin{proof}
The first step of $\pi(o)$ must visit a neighbour of $o$, denoted by $w$, 
then $P[\bar{A}_o^c]$
\begin{equation*}
\begin{split}
&= \Pr[\text{The second step of $\pi(o)$ visits $x\neq 2w$, 
 the third step visits $w' \sim o$, $w'\neq w$}] + O\Big(\frac{1}{d^2}\Big)\\
&=\Big(\frac{1}{2d}\Big)\Big(\frac{2d-1}{2d}\Big)
 + O\Big(\frac{1}{d^2}\Big)
= O\Big(\frac{1}{d}\Big).\\
\end{split}
\end{equation*} 
\end{proof}

Let $\bar{A}_{all} = \{ \forall w \sim o : \text{either $\pi(w)$ does not visit $o$ 
or $\pi(w)$ visits $o$ at the first step }\}.$

\begin{lemma}
\label{lem:barAall}
$\Pr[\exists w\sim o: \text{$\pi(w)$ visits $o$ but not at the first step}] 
= \Pr[\bar{A}_{all}^c] = O(1/d).$
\end{lemma}

\begin{proof}
For a given $w$, $w \sim o$, use Wilson's algorithm with a walk started at $w$. 
Consider that if $S_1^w \neq o$, or $S_1^w = o$ but $S^w$ returns to $w$ subsequently 
and then this loop starting from $w$ in $S^w$ is erased, $\pi(w)$ does not visit $o$ 
at the first step. Hence, we have the inequality:
\begin{equation}
\label{e:pi(w)}
\begin{split}
&\Pr[\text{$\pi(w)$ visits $o$ but not at the first step}] \\
&\leq \Pr[\text{$S^w$ visits $o$ but not at the first step}]
 + \Pr[S_1^w = o, S_n^w = w \text{ for some $n\geq 2$}].
\end{split}
\end{equation}
We bound the two terms as follows.
For the first term, let us append a step from $o$ to $w$ at the beginning
of the walk, and analyze it as if the walk started at $o$. Since
$S_1^o \in \mathcal{N} \backslash \{o\}$, by symmetry, we may assume $S_1^o = w$. 
Then if $S_2^o \neq o$, $S^o$ will need at least $2$ more steps to return to $o$.
 
For the second term in the right hand side of \eqref{e:pi(w)}, we first 
note that we have $\Pr[S_1^w = o, S_2^w = w] = 1/(2d)^2$. If $S^w$ does not return to $w$ 
in the first two steps, $S^w$ will need at least $4$ steps to return to $w$. 
Then, we have that the right hand side of \eqref{e:pi(w)} is 
\begin{equation*}
\begin{split}
&\leq \Pr[\text{$S^o$ returns to $o$ in at least $4$ steps}] 
+ \frac{1}{(2d)^2} + \Pr[\text{$S^w$ returns to $w$ in at least $4$ steps}]\\
&= 2\times\Pr[\text{$S^o$ returns to $o$ in at least $4$ steps}] + O\Big(\frac{1}{d^2}\Big).
\end{split}
\end{equation*}

Therefore, by Lemma \ref{lem:returntoo}, we have the required result
\begin{equation*}
\begin{split}
&\Pr[\exists w\sim o: \text{$\pi(w)$ visits $o$ but not at the first step }]\\
& = 2d\times\Pr[\text{$\pi(w)$ visits $o$ but not at the first step for a fixed $w \sim o$}]
=O\Big(\frac{1}{d}\Big).
\end{split}
\end{equation*}
\end{proof}

Due to Lemmas \ref{lem:barAo} and \ref{lem:barAall}, we have 
\begin{equation*}
q_d(i) 
\leq O\Big(\frac{1}{d}\Big) + \Pr[\bar{A}_o \cap \bar{A}_{all} \cap E_i]
= O\Big(\frac{1}{d}\Big) + \sum_{\substack{x_1,\dots,x_i \sim o \\ \text{distinct}}} 
    \Pr[\bar{A}_o \cap \bar{A}_{all} \cap E_i(x_1,\dots,x_i)]. 
\end{equation*}
Here, 
\begin{equation}
\label{e:AoAallEi}
\begin{split}
&\bar{A}_o \cap \bar{A}_{all} \cap E_i(x_1,\dots,x_i) \\
&\subset \bar{A}_o \cap \bar{A}_{all} \cap \{\text{the first step of $\pi(x_j)$ 
  is to $o$, $j = 1,\dots,i$}\} \cap F_i(x_1,\dots,x_i),
\end{split}
\end{equation}
where 
\begin{equation*}
F_i(x_1,\dots,x_i) 
= \{\text{$\pi(x_j)$ does not go through $o$, $j = i+1, \dots,2d$}\}.
\end{equation*}
The right hand side of \eqref{e:AoAallEi} is contained in the event
\begin{equation*}
\bar{A}_o \cap \{\text{$\pi(o)$ does not visit $x_1,\dots,x_i$}\} 
\cap \bar{A}_{rest} \cap \bigcap_{1\leq j \leq i} H_j \cap F_i(x_1, \dots, x_i),
\end{equation*}
where 
\begin{equation*}
\bar{A}_{rest}
= \{\text{$\pi(x_j)$ goes through at most one $x_{j'}$, 
  $j=i+1,\dots,2d$, $i+1 \leq j' \leq 2d$, $j' \neq j$} \}
\end{equation*}
and $H_j = \{\text{the first step of $\pi(x_j)$ is to $o$}\}$ for $j = 1,\dots,i$.

We denote $\bar{A}_o\cap\{\text{$\pi(o)$ does not visit $x_1,\dots,x_i$}\}$ 
by $\bar{A}_{o,x_1,\dots,x_i}$. Then
\begin{equation*}
\begin{split}
&\Pr\Big[\bar{A}_{o,x_1,\dots,x_i} \cap \bar{A}_{rest} 
  \cap\bigcap_{1\leq j \leq i} H_j \cap F_i(x_1, \dots, x_i)\Big] \\
&=\Pr[\bar{A}_{o,x_1,\dots,x_i}] 
  \prod_{j=1}^i\Pr\Big[H_j\Big|\bigcap_{1\leq j' < j} H_{j'} 
  \cap \bar{A}_{o,x_1,\dots,x_i}\Big] \\
&\quad\times\Pr\Big[F_i(x_1,\dots,x_i)\cap\bar{A}_{rest}
  \Big|\bar{A}_{o,x_1,\dots,x_i}\cap\bigcap_{1\leq j \leq i}H_j\Big].
\end{split}
\end{equation*}

Therefore, we have
\begin{equation}
\label{e:qi}
\begin{split}
q_d(i)
&\leq O\Big(\frac{1}{d}\Big)
+\sum_{\substack{x_1,\dots,x_i\sim o \\ \mathrm{distinct}}}  
 \Big(\prod_{j=1}^i\Pr\Big[H_j
  \Big|\bigcap_{1\leq j' < j}H_{j'}\cap\bar{A}_{o,x_1,\dots,x_i}\Big]\Big) \\
&\quad\quad \times\Pr\Big[F_i(x_1,\dots,x_i)\cap\bar{A}_{rest}
  \Big|\bar{A}_{o,x_1,\dots,x_i}\cap\bigcap_{1\leq j \leq i}H_j\Big].
\end{split}
\end{equation}

\begin{lemma}
\label{lem:Hj}
$\Pr[H_j | \bar{A}_{o,x_1,\dots,x_i} \cap \bigcap_{1\leq j' < j}H_{j'}] 
= 1/2d + O(1/d^2)$, where $j = 1,\dots,i$.
\end{lemma}

\begin{proof}
Given that $\pi(o)$ visits only one neighbour of $o$ which is not in $\{x_1,\dots,x_i\}$ 
and the first steps of $\pi(x_1),\dots,\pi(x_{j-1})$ are all to $o$, the probability 
that $H_j$ happens is $\Pr[S_1^{x_j} = o] = 1/2d$ with the error term of $O(1/d^2)$ 
due to the loop-erasure.
\end{proof}

\begin{lemma}
\begin{equation}
\label{e:FiArest}
\Pr\Big[F_i(x_1,\dots,x_i)\cap \bar{A}_{rest}
 \Big|\bar{A}_{o,x_1,\dots,x_i}\cap\bigcap_{1\leq j \leq i}H_j\Big]
\leq\E\Big[\Big(1-\frac{1}{2d}+O\Big(\frac{1}{d^2}\Big)\Big)^{2d-i-1-N}
 \mathbf{1}_{\bar{A}_{rest}}\Big],
\end{equation}
where 
$N=|\{i+1\leq j\leq 2d-1:\text{$\exists i+1\leq j'<j$ s.t. 
$\pi(x_{j'})$ goes through $x_j$}\}|$.
\end{lemma}

\begin{proof}
Consider Wilson's algorithm with random walks started at the remaining neighbours 
$x_{i+1}, \dots,x_{2d}$. 
Assume $x_{2d}$ to be the neighbour of $o$ that $\pi(o)$ goes through.
The probability that $\pi(x_k)$ does not go through $o$ is $1 - 1/2d + O(1/d^2)$ 
for $k\in \{i+1,\dots,2d-1\}$. 

If $\pi(x_k)$ visits $x_{k'}$, where $k < k' \leq 2d-1$, 
the probability that $\pi(x_{k'})$ does not go through $o$ is $1$ instead of 
$1-1/2d+O(1/d^2)$, since the LERW from $x_{k'}$ stops immediately 
and $\pi(x_{k'}) \subset \pi(x_k)$, which does not go through $o$.
\end{proof}

\begin{lemma}
\label{lem:Bin}
On the event $\bar{A}_{rest}$, $N \leq B$, 
where $B \sim \mathsf{Binom}(2d-i-1, p)$, $p = 1/2d + O(1/d^2)$.
\end{lemma}

\begin{proof}
Since we have $(2d-i-1)$ trials with probability at most $1/2d+O(1/d^2)$.
\end{proof}

Due to Lemma \ref{lem:Bin}, we have that the right hand side of 
\eqref{e:FiArest} is
\begin{equation}
\label{e:FiArest1}
\leq \Big(1-\frac{1}{2d}+O\Big(\frac{1}{d^2}\Big)\Big)^{2d}
 \Big(1+O\Big(\frac{i}{d}\Big)\Big)
 \E\Big[\frac{1}{(1-\frac{1}{2d}+O(\frac{1}{d^2}))^B}\Big],
\end{equation}
where
$\E[z^B] 
= \sum_{j=0}^{2d-i-1} z^j {2d-i-1 \choose j}p^j(1-p)^{2d-i-1-j} 
= (1-p-zp)^{2d-i-1}$.

Hence \eqref{e:FiArest1} is
\begin{equation}
\label{e:FiArest2}
\begin{split}
&\leq e^{-1}\Big(1+O\Big(\frac{1}{d}\Big)\Big)\Big(1+O\Big(\frac{i}{d}\Big)\Big)
 \Big(1-\frac{1}{2d} + O\Big(\frac{1}{d^2}\Big) + \frac{\frac{1}{2d} 
 + O(\frac{1}{d^2})}{1-\frac{1}{2d} 
 + O(\frac{1}{d^2})}\Big)^{2d-i-1}\\
&= e^{-1}\Big(1+O\Big(\frac{1}{d}\Big)\Big)
   \Big(1+O\Big(\frac{i}{d}\Big)\Big)\Big(1+O\Big(\frac{1}{d^2}\Big)\Big)^{2d-i-1}
= e^{-1}\Big(1+O\Big(\frac{i}{d}\Big)\Big).
\end{split}
\end{equation}

\begin{lemma}
\label{lem:upperb}
$q_d(i) \leq O(\frac{1}{d}) + e^{-1}\frac{1}{i!}(1+O(\frac{i}{d})).$
\end{lemma}

\begin{proof}
Due to Lemma \ref{lem:Hj}, \eqref{e:qi} and \eqref{e:FiArest2}, 
we have
\begin{equation*}
\begin{split}
q_d(i)
&\leq O\Big(\frac{1}{d}\Big)
+ {2d\choose i}\Big(\frac{1}{2d}+O\Big(\frac{1}{d^2}\Big)\Big)^i
  e^{-1}\Big(1+O\Big(\frac{i}{d}\Big)\Big)\\
& = O\Big(\frac{1}{d}\Big)
   + e^{-1}\frac{2d(2d-1)\dots(2d-i+1)}{i!}
   \Big(\frac{1}{2d}\Big)^i\Big(1+O\Big(\frac{1}{d}\Big)\Big)^i
   \Big(1+O\Big(\frac{i}{d}\Big)\Big)\\
&\leq O\Big(\frac{1}{d}\Big) 
   + e^{-1}\frac{1}{i!}\Big(1+O\Big(\frac{i}{d}\Big)\Big).
\end{split}
\end{equation*}
\end{proof}

\begin{lemma}
\label{lem:pio}
For $k = 1, \dots, 3$ and distinct $w_1, \dots, w_k \sim o$, we have
\begin{equation*}
  \Pr [ \text{$\pi(w_i)$ passes through $o$ for $i = 1,\dots,k$} ]
  = \left( \frac{1}{2d} \right)^k + O (d^{-k-1}). 
\end{equation*}
\end{lemma}
This lemma can be proved using ideas used to prove Lemma \ref{lem:Hj}.

\subsection{Proof of the asymptotic formula}

\begin{proof}[Proof of Theorem \ref{thm:asymformula}]
We first prove part (i).
By Wilson's algorithm, 
\begin{equation*}
p_d(i) = \sum_{j=0}^{i} \frac{q_d(j)}{2d-j}.
\end{equation*}
Due to Lemmas \ref{lem:lowerb} and \ref{lem:upperb} , we have
\begin{equation}
\label{e:pilower}
p_d(i)
\geq\sum_{j=0}^{i} \frac{e^{-1}\frac{1}{j!}(1 + O(\frac{j^2}{d}))}{2d-j}
=\sum_{j=0}^{i}\frac{e^{-1}\frac{1}{j!}}{2d-j} 
 +\sum_{j=0}^{i} \frac{\frac{1}{j!} O(\frac{j^2}{d})}{2d-j},
\end{equation}
and 
\begin{equation}
\label{e:piupper}
p_d(i)
\leq\sum_{j=0}^{i}\frac{O(\frac{1}{d}) + e^{-1}\frac{1}{j!}(1+O(\frac{j}{d}))}{2d-j} 
=\sum_{j=0}^{i}\frac{e^{-1}\frac{1}{j!}}{2d-j} 
+\sum_{j=0}^{i}\frac{O(\frac{1}{d})+\frac{1}{j!}O(\frac{j}{d})}{2d-j}.
\end{equation}
Here, using that $0\le j\le d^{1/2}$, we have
\begin{equation*}
\sum_{j=0}^{i} \frac{\frac{1}{j!} O(\frac{j^2}{d})}{2d-j}
\leq \frac{1}{2d-d^{1/2}}O(\frac{1}{d})\sum_{j=0}^{i}\frac{j^2}{j!}
=O(\frac{1}{d^2}).
\end{equation*}
Similarly,
\begin{equation*}
\begin{split}
\sum_{j=0}^{i} \frac{O(\frac{1}{d}) + \frac{1}{j!} O(\frac{j}{d})}{2d-j}
\le \sum_{j=0}^i O(d^{-2}) + \sum_{j=0}^i \frac{j}{j!} O(d^{-2})
= O( i/d^2 ).
\end{split}
\end{equation*}
Putting these error bounds together with \eqref{e:pilower} 
and \eqref{e:piupper}, we prove statement (i) of the theorem.

Let us now use that 
\begin{equation*}
\frac{1}{2d}e^{-1}\sum_{j=0}^{i} \frac{1}{j!} 
\leq \sum_{j=0}^{i} \frac{e^{-1} \frac{1}{j!}}{2d-j} 
\leq \frac{1}{2d-i}e^{-1}\sum_{j=0}^{i}\frac{1}{j!}.
\end{equation*} 
When $i \le d^{1/2}$, and $i, d \rightarrow \infty $,
we have $\frac{1}{2d-i} \sim \frac{1}{2d}$  
and $\sum_{j=0}^{i}\frac{1}{j!} \rightarrow e$.
Hence,
\begin{equation*}
\sum_{j=0}^{i} \frac{e^{-1} \frac{1}{j!}}{2d-j} \sim \frac{1}{2d},
\quad\text{as $i, d \rightarrow \infty$}.
\end{equation*}

We are left to prove statement (ii).
The uniform distribution for $d^{1/2} \le i \le 2d-1$ can be obtained from the 
monotonicity:
\begin{equation*}
p_d(d^{1/2}) \le p_d(i) \le p_d(2d-1), \quad d^{1/2} \le i \le 2d-1, 
\end{equation*}
if we show that $p_d(2d-1) = p_d(d^{1/2}) + O(d^{-3/2})$.

We write 
\begin{equation}
\label{e:pd2d-1}
p_d(2d-1)
= \sum_{j=0}^{2d-1} \frac{q_d(j)}{2d-j}
= p_d(d^{1/2}) + \sum_{j = d^{1/2}}^{2d-1} \frac{q_d(j)}{2d-j}
\le p_d(d^{1/2}) + \sum_{j = d^{1/2}}^{2d-1} q_d(j).
\end{equation}

Introducing the random variable
\begin{equation*}
X:= |\left\{ w \sim o : o \in \pi(w) \right\}|,
\end{equation*}
the last expression in \eqref{e:pd2d-1} equals
\begin{equation*}
p_d(d^{1/2}) + \Pr [ X \ge d^{1/2} ]
\le p_d(d^{1/2}) + \Pr [ X^3 \ge d^{3/2} ]
\le p_d(d^{1/2}) + \frac{\E [ X^3 ]}{d^{3/2}}. 
\end{equation*}
Therefore, it remains to show that $\E[ X^3 ] = O(1)$.
This follows from Lemma \ref{lem:pio}, by summing over $w_1, \dots, w_3$ 
(not necessarily distinct). The cases $k=1,2$ of the lemma are used to 
sum the contributions where one or more of the $w_i$'s coincide.
\end{proof}

\bibliographystyle{plain}
\bibliography{height_probability_revision}

\begin{thebibliography}{10}

\bibitem{Aldous1990}
D~J Aldous.
\newblock The random walk construction of uniform spanning trees and uniform
  labelled trees.
\newblock {\em SIAM Journal on Discrete Mathematics}, 3(4):450--465, 1990.

\bibitem{AthreyaJarai2004}
S~R Athreya and A~A J\'{a}rai.
\newblock Infinite volume limit for the stationary distribution of abelian
  sandpile models.
\newblock {\em Communications in Mathematical Physics}, 249(1):197--213, 2004.

\bibitem{BakTangWiesenfeld1987}
P~Bak, C~Tang, and K~Wiesenfeld.
\newblock Self-organized criticality: An explanation of the 1/ f noise.
\newblock {\em Physical Review Letters}, 59(4):381--384, 1987.

\bibitem{BenjaminiLyonsPeresSchramm2001}
I~Benjamini, R~Lyons, Y~Peres, and O~Schramm.
\newblock Uniform spanning forests.
\newblock {\em Annals of Probability}, 29(1):1--65, 2001.

\bibitem{Dhar1990}
D~Dhar.
\newblock Self-organized critical state of sandpile automaton models.
\newblock {\em Physical Review Letters}, 64(14):1613--1616, 1990.

\bibitem{Dhar2006}
D~Dhar.
\newblock Theoretical studies of self-organized criticality.
\newblock {\em Physica A: Statistical Mechanics and its Applications},
  369(1):29--70, 2006.

\bibitem{DharMajumdar1990}
D~Dhar and S~N Majumdar.
\newblock Abelian sandpile model on the {B}ethe lattice.
\newblock {\em Journal of Physics A: Mathematical and General},
  23(19):4333--4350, 1990.

\bibitem{GamlinJarai2014}
S~L Gamlin and A~A J\'{a}rai.
\newblock Anchored burning bijections on finite and infinite graphs.
\newblock {\em Electronic Journal of Probability}, 19(117):1--23, 2014.

\bibitem{HolroydLevineMeszarosPeresProppWilson2008}
A~E Holroyd, L~Levine, K~M{\'e}sz{\'a}ros, Y~Peres, J~Propp, and D~B Wilson.
\newblock Chip-firing and rotor-routing on directed graphs.
\newblock In V~Sidoravicius and M~E Vares, editors, {\em In and Out of
  Equilibrium 2}, pages 331--364. Birkh{\"a}user, Basel, 2008.

\bibitem{Jarai2018}
A~A J\'arai.
\newblock Sandpile models.
\newblock {\em Probability Surveys}, 15(0):243--306, 2018.

\bibitem{JaraiSun2019a}
A~A J\'arai and M~Sun.
\newblock Toppling and height probabilities in sandpiles.
\newblock {\em To appear in J. Stat. Mech. Theory Exp.}, 2019.

\bibitem{JaraiWerning2014}
A~A J\'{a}rai and N~Werning.
\newblock Minimal configurations and sandpile measures.
\newblock {\em Journal of Theoretical Probability}, 27(1):153--167, 2014.

\bibitem{JengPirouxRuelle2006}
M~Jeng, G~Piroux, and P~Ruelle.
\newblock Height variables in the abelian sandpile model: scaling fields and
  correlations.
\newblock {\em Journal of Statistical Mechanics: Theory and Experiment},
  2006(10):P10015, 2006.

\bibitem{KenyonWilson2015}
R~W Kenyon and D~B Wilson.
\newblock Spanning trees of graphs on surfaces and the intensity of loop-erased
  random walk on planar graphs.
\newblock {\em Journal of the American Mathematical Society}, 28(4):985--1030,
  2015.

\bibitem{LevinePeres2017}
L~Levine and Y~Peres.
\newblock Laplacian growth, sandpiles, and scaling limits.
\newblock {\em Bulletin of the American Mathematical Society}, 54(3):355--382,
  2017.

\bibitem{LyonsPeres2016}
R~Lyons and Y~Peres.
\newblock {\em Probability on trees and networks}.
\newblock Cambridge University Press, Cambridge, 2016.

\bibitem{MadrasSlade2013}
N~Madras and G~Slade.
\newblock {\em The Self-Avoiding Walk}.
\newblock Springer New York, 2013.

\bibitem{MajumdarDhar1991}
S~N Majumdar and D~Dhar.
\newblock Height correlations in the abelian sandpile model.
\newblock {\em Journal of Physics A: Mathematical and General},
  24(7):L357--L362, 1991.

\bibitem{MajumdarDhar1992}
S~N Majumdar and D~Dhar.
\newblock Equivalence between the abelian sandpile model and the $q \to 0$
  limit of the potts model.
\newblock {\em Physica A: Statistical Mechanics and its Applications},
  185(1):129--145, 1992.

\bibitem{PapoyanShcherbakov1996}
V~V Papoyan and R~R Shcherbakov.
\newblock Distribution of heights in the abelian sandpile model on the {H}usimi
  lattice.
\newblock {\em Fractals}, 4(1):105--110, 1996.

\bibitem{PoghosyanPriezzhevRuelle2011}
V~S Poghosyan, V~B Priezzhev, and P~Ruelle.
\newblock Return probability for the loop-erased random walk and mean height in
  the abelian sandpile model: a proof.
\newblock {\em Journal of Statistical Mechanics: Theory and Experiment},
  2011(10):P10004, 2011.

\bibitem{Priezzhev1994}
V~B Priezzhev.
\newblock Structure of two-dimensional sandpile. {I}. {H}eight probabilities.
\newblock {\em Journal of Statistical Physics}, 74(5):955--979, 1994.

\bibitem{Redig2006}
F~Redig.
\newblock Mathematical aspects of the abelian sandpile model.
\newblock In {\em Mathematical statistical physics: Lecture Notes of the Les
  Houches Summer School 2005}, pages 657--729. Elsevier, 2006.

\bibitem{Wilson1996}
D~B Wilson.
\newblock Generating random spanning trees more quickly than the cover time.
\newblock In {\em Proceedings of the Twenty-eighth Annual ACM Symposium on
  Theory of Computing}, STOC '96, pages 296--303, New York, NY, USA, 1996. ACM.

\end{thebibliography}

\end{document}